\documentclass[runningheads]{llncs}
\usepackage[T1]{fontenc}
\usepackage{amsmath, amssymb, mathrsfs, multirow, url, subfigure}
\usepackage{ifthen} 
\usepackage{amsfonts}
\usepackage{tabularray}
\newcommand{\unif}{{\sf Unif}}
\newcommand{\betadist}{{\sf Beta}}

\usepackage{graphicx}
\usepackage[sort&compress]{natbib}
\bibpunct{(}{)}{;}{a}{}{,} 
\makeatletter
\renewcommand\bibsection%
{
  \section*{Bibliography}
}
\makeatother
\begin{document}
\title{Fusing independent inferential models in a black-box manner}
\titlerunning{Fusing independent IMs}
%
\author{Leonardo Cella}
\authorrunning{L. Cella}
%
\institute{Wake Forest University \\ Department of Statistical Sciences \\ Winston-Salem, North Carolina, 27109, U.S.A \\
\email{cellal@wfu.edu}\\
}
\maketitle              
\begin{abstract}
Inferential models (IMs) represent a novel possibilistic approach for achieving provably valid statistical inference.
This paper introduces a general framework for fusing independent IMs in a ``black-box'' manner, requiring no knowledge of the original IMs construction details. The underlying logic of this framework mirrors that of the IMs approach. First, a fusing function for the initial IMs' possibility contours is selected. Given the possible lack of guarantee regarding the calibration of this function for valid inferences, a ``validification'' step is performed. Subsequently, a straightforward normalization step is executed to ensure that the final output conforms to a possibility contour.

\end{abstract}
\section{Introduction}
The art of combining different sources of evidence to distinguish signal from noise remains highly relevant in today's world, where data is more abundant than ever and manifests in various forms. Take for example independent studies that were conducted to answer a common question. Each study may yield a measure of evidence for the quantity of interest. The objective is to integrate these specific measures of evidence into a composite measure that offers greater insight. Note that, in these scenarios, ``data'' refers to the measures of evidence themselves.


Along these lines, the present paper is primarily focused on exploring new developments in the combination/fusion of independent {\em inferential models} (IMs) \citep{martinbook}. IMs constitute a contemporary statistical inference framework that outputs provably valid possibility and necessity measures to any claim of interest concerning the inferential target. 
The possibilistic nature of IMs entails the existence of a possibility contour that is the base for all IMs uncertainty quantification. In what follows, the possibility contours from independent IMs will constitute the ``data'' to be combined.

Specifically, the goal is to fuse the independent IMs' possibility contours in a way that the resulting function maintains the calibrated possibilistic characteristic of an IM while exhibiting favorable efficiency properties. Moreover, this fusion should not require prior knowledge of how the original contours were constructed. 

There exists a vast literature addressing the fusion of possibility contours. An extensive overview can be found in \citet{Dubois1999}, but see also \citet{DuboisPrade88,dubois:hal-01484952,Dubois2001}. However, these fusing rules typically do not take calibration assurances into consideration.
Given that validity is the primary pillar of IMs, an alternative fusion strategy is needed.

After providing a brief background on the IMs approach in Section~\ref{ss:Back}, I outline the proposed methodology in Section~\ref{S:main}. I begin by examining an intuitive fusion alternative in Section~\ref{ss:First}, namely taking the minimum of the IMs' possibility contours. However, this solution turns out to be unsatisfactory for two reasons: it lacks calibration and does not conform to a possibility contour.  In Section~\ref{ss:ValNorm}, I propose remedial measures to address these shortcomings. These remedial measures are not specific to the initial solution. In fact, they align well with the overall logic of the IMs approach. Consequently, in Section~\ref{ss:general}, I propose a general solution to the fusion of independent IMs. It is noteworthy that this general solution not only shares connections with the fusion methods in possibility theory mentioned above, but also exhibits similarities to methods proposed for merging independent p-values in the statistics literature \citep[e.g.,][]{OosterhoffPvalues,OwenPvalue,cousins2008annotated}.
A concise summary and some important remarks are provided in Section~\ref{S:Conclusion}.


\section{Background on IMs}
\label{ss:Back}
Let $Y^n = (Y_1, \ldots, Y_n)$ represent $n$ independent random draws from a distribution $\mathbb{P}_\theta$ that is associated with some unknown quantity of interest $\theta \in \Theta$. This setup is supposed to be general, encompassing scenarios where $\mathbb{P}_\theta$ conforms to a parametric model, as well as instances where the distribution of $Y^n$ remains unspecified, with $\theta$ serving to characterize a feature of this distribution. The goal is to perform probabilistic inferences on $\theta$ in the sense that, after observing some data $Y^n=y^n$, degrees of belief  can be assigned to claims of interest about $\theta$.

IMs constitute a relatively novel framework for probabilistic inference, employing necessity measures to quantify degrees of belief. What sets IMs apart from other probabilistic methodologies, whether grounded in precise or imprecise probabilities, is the verifiable calibration assurance of its degrees of belief. 
Specifically, if $\mathcal{N}_{y^n}(A)$ represents the IMs' necessity measure attributed to the claim ``$\theta \in A$'' after observing data $Y^n=y^n$, the following {\em validity property}
\begin{equation}\label{eq:NeceVal}
\sup_{\theta \notin A}\mathbb{P}_\theta^n \left\{ \mathcal{N}_{Y^n}(A) \geq 1-\alpha \right \} \leq \alpha, \quad  \text{for all $\alpha \in [0,1]$ and all $A \subset \Theta$}
\end{equation}
is satisfied. In words, \eqref{eq:NeceVal} states that the assignment of high degree of belief to false claims about $\theta$ is a rare event with respect to $Y^n \sim \mathbb{P}_\theta^n$. 


The possibilistic nature of IMs entails the existence of a {\em possibility contour}, a function $\pi_{y^n}(\vartheta)$ on $(\mathbb{Y}^n \times \Theta) \rightarrow [0,1]$ satisfying 
\begin{equation}\label{eq:consonance}
 \sup_{\vartheta \in \Theta} \pi_{y^n}(\vartheta)=1 \quad \text{for all } y^n.   
\end{equation}
This possibility contour serves as the foundation for computing necessity measures for claims of interest within IMs. Moreover, a necessary and sufficient condition for the IMs validity in \eqref{eq:NeceVal} is that the possibility contour is stochastically no smaller than a $\unif(0,1)$ distribution. This can be stated as
\begin{equation}\label{eq:IMvalidity}
\mathbb{P}_\theta^n\{ \pi_{Y^n}(\theta) \leq \alpha\} \leq \alpha, \quad  \text{for all $\alpha \in [0,1]$.}
\end{equation}
An IM is termed {\em exact} valid when $\pi_{Y^n}(\theta) \sim \unif(0,1)$. 

The ``stochastically no smaller than uniform'' outlined in \eqref{eq:IMvalidity} is a well known characteristic of p-values. However, it is important to emphasize that the IMs' possibility contours, besides satisfying such criteria, also satisfy \eqref{eq:consonance}, a characteristic that p-values don't always satisfy. As studied in detail in \citet{imprecisefrequentist}, the reason to constrain IMs to possibility measures is twofold. First, for any probabilistic approach with valid degrees of belief in the sense of \eqref{eq:NeceVal}, there exists a possibilistic IM whose necessity measure is no less efficient. Secondly, possibility and necessity measures stand out among the most computationally straightforward forms of imprecise probabilities.



But how are IMs constructed? Here I'll focus on a modern construction presented in \citet{ryanpp1,ryanpp2}. Suppose one has access to a measurable function $\gamma_{y^n}(\vartheta)$ on $(\mathbb{Y}^n \times \Theta) \rightarrow [0,1]$  that orders candidate values for $\theta$ in a ``plausibilistic'' manner given the observed data $Y^n=y^n$. That is, this function distinguishes between candidate values for $\theta$ that are more or less plausible given $y^n$, where values closer to zero and one indicate disagreement and agreement, respectively. It is usually the case that the chosen $\gamma_{y^n}$ satisfies $\sup_{\vartheta \in \Theta} \gamma_{y^n}(\vartheta)=1$ for all $y^n$, as seen in the natural choice of the likelihood ratio for $\gamma_{y^n}$ when $\mathbb{P}_\theta$ is a parametric model. The properties of a possibility contour are satisfied for such choices, suggesting that the IM construction is complete and enabling the computation of necessity measures for claims of interest concerning $\theta$ to proceed. However, there is no guarantee that $\gamma_{y^n}$ satisfies the validity criteria in \eqref{eq:IMvalidity} 
so it undergoes a {\em validification} step:
\begin{equation}\label{eq:contourIMpar}
\pi_{y^n}(\vartheta) = \mathbb{P}_\vartheta^n\{\gamma_{Y^n}(\vartheta) \leq \gamma_{y^n}(\vartheta)\}, \quad \vartheta \in \Theta,
\end{equation}
giving rise to a new possibility contour that is guaranteed to satisfy \eqref{eq:IMvalidity}. To see this, note that $\pi_{y^n} = G(\gamma_{y^n})$, where $G$ is the distribution function of $\gamma_{Y^n}$, so $\pi_{Y^n}(\theta) = G(\gamma_{Y^n}(\theta)) \sim \unif(0,1)$.

\section{Fusing IMs}
\label{S:main}
\subsection{Setup, objectives and a first idea}
\label{ss:First}

Let $\pi_1, \ldots, \pi_k$ represent $k$ independent IMs' possibility contours, each constructed to make valid probabilistic inferences on an unknown quantity $\theta$. While each IM may be based on a different sample size, I will assume here that this information is not available, so the notation for the IMs' contours suppresses sample sizes for the remainder of the paper. It will be clear below that the sample sizes are not even necessary for the proposed methodology, but see Section~\ref{S:Conclusion}. The goal is to fuse the $k$ possibility contours in a way that 
\begin{enumerate}
    \item the resulting function is also a valid possibility contour;
    \item favorable efficiency properties are exhibited;
    \item no knowledge of how each contour was constructed is required.    
\end{enumerate}

The first requirement is straightforward, meaning that the solution must satisfy properties \eqref{eq:consonance} and \eqref{eq:IMvalidity}, so that it can be used later on to make valid probabilistic inferences about $\theta$. The second requirement is intuitively reasonable. Since we are combining $k$ independent sources of information about $\theta$ we ideally want the solution to show efficiency gains. 
The third requirement means that the fusion of IMs will be done in a black-box manner. 

I start by considering a very intuitive approach to fusing independent IMs. With a focus on ensuring the efficiency of the solution, it involves taking the minimum' of the $k$ contours:
\[\gamma_{min}(\vartheta) = \min\{\pi_1(\vartheta), \ldots, \pi_k(\vartheta)\}, \quad \vartheta \in \Theta. \]
While $\gamma_{min}$ won't exceed any of the $k$ original contours in size, its validity is uncertain. Moreover, there's no guarantee that $\gamma_{min}$ satisfies \eqref{eq:consonance}, meaning that $\gamma_{min}$ may not necessarily qualify as a possibility contour. This can be seen Figure~\ref{fig:firstideas}(b), where $\gamma_{min}$ is displayed for the contours depicted in Figure~\ref{fig:firstideas}(a).


\begin{figure}[t]
\begin{center}
\subfigure[IMs' contours to be fused]{\scalebox{0.3}{\includegraphics{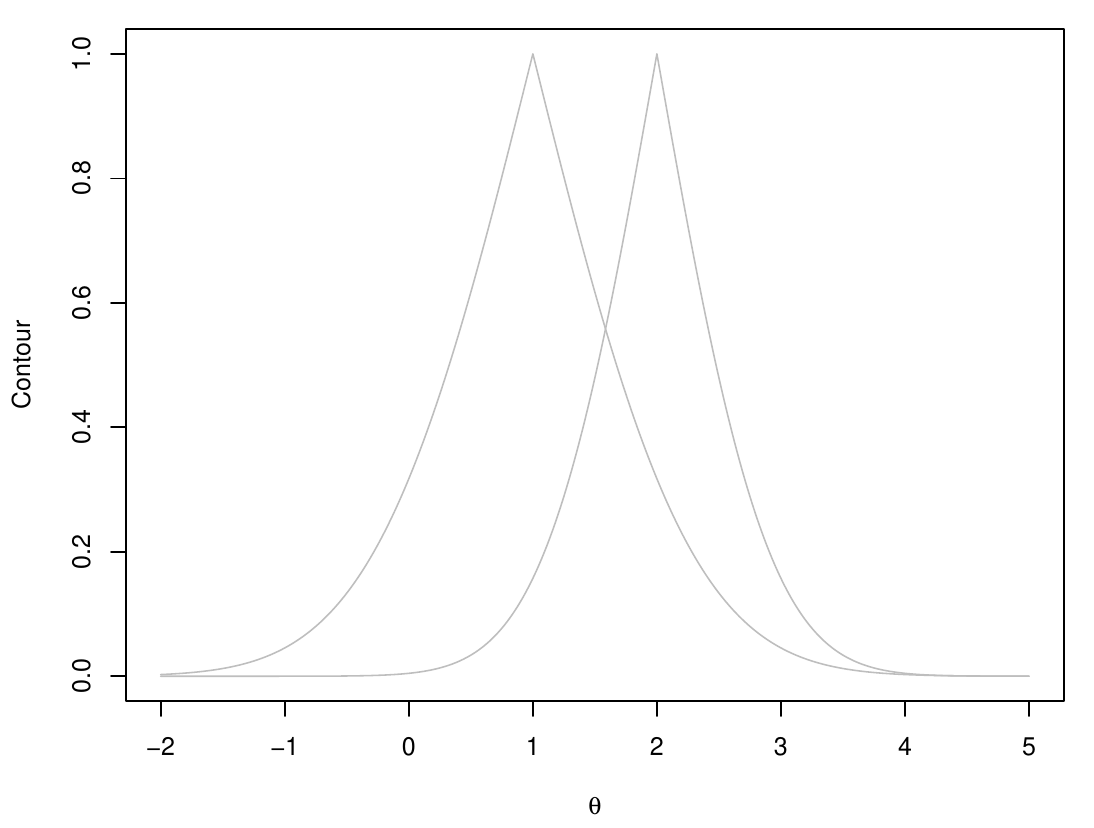}}}
\subfigure[$\gamma_{min}$, $\pi_{min}$ and $\dot{\pi}_{min}$]{\scalebox{0.3}{\includegraphics{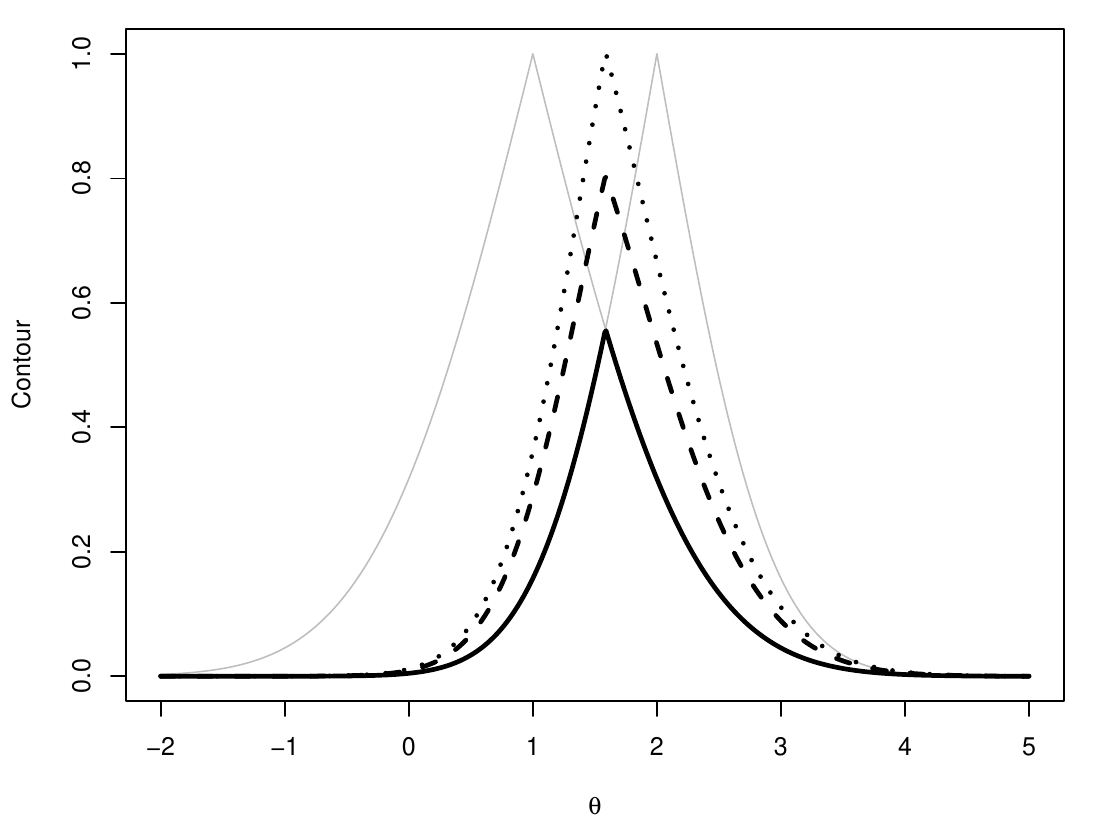}}}
\end{center}
\caption{Details from the example in Sections~\ref{ss:First} and \ref{ss:ValNorm}. The solid, dashed and dotted black lines in Panel~(b) 
represent $\gamma$, $\pi$ and $\dot{\pi}$, respectively.}
\label{fig:firstideas}
\end{figure}

\subsection{Validification and normalization}
\label{ss:ValNorm}
Considering $\gamma_{\text{min}}$ by itself is not an entirely satisfactory solution for fusing independent IMs. However, I will argue here that there is still potential in this solutions. More specifically, there is an opportunity to refine and make it desirable. And the means to do so align well with the general logic of the IMs framework, presented in Section~\ref{ss:Back}. For the rest of this section, $\pi_i, i=1, \ldots k,$ and $\gamma_{min}$ will be denoted by $\Pi_i$ and $\Gamma_{min}$, respectively, when representing random variables. 
Moreover, I will assume $\pi_1, \ldots, \pi_k$ are contours of exact valid IMs, so $\Pi_1(\theta), \ldots$, $\Pi_k(\theta)$ are independent and identically distributed (iid) $\unif(0,1)$.

The issues with $\gamma_{min}$ were twofold: it isn't guaranteed to satisfy the validity property outlined in \eqref{eq:IMvalidity}, and it doesn't necessarily qualify as a possibility contour. For the lack of validity, remember from Section~\ref{ss:Back} that, in the IMs construction, the selected function to plausibilistically rank candidate values of the unknown parameter given observed data might lack calibration, despite its intuitive appeal. This is precisely why the validification step in \eqref{eq:contourIMpar} is performed. Here, the concept remains the same, and $\gamma_{\min}$ will undergo validification.

Towards this, recall the well known result in probability theory that if $V$ is a random variable defined as the minimum of $k$ independent $\unif(0,1)$ random variables, then $V \sim \betadist(1,k)$, a distribution that is stochastically smaller than $\unif(0,1)$ for every $k$. This confirms the lack of validity of $\gamma_{min}$ since $\Gamma_{min}(\theta) \sim \betadist(1,k)$. 
However, by following the structure of \eqref{eq:contourIMpar}, $\gamma_{min}$ can be validified:
\[\pi_{min}(\vartheta) = F(\gamma_{min}(\vartheta)), \quad \vartheta \in \Theta,\]
where $F$ is the distribution function of $\betadist(1,k)$. Validity of $\pi_{min}$ comes from the fact that $F(\Gamma_{min}(\theta))\sim \unif(0,1)$.


Unfortunately, validification alone doesn't automatically convert $\gamma_{min}$ into a possibility contour. This is depicted in Figure~\ref{fig:firstideas}(b), where both $\pi_{min}$ and $\gamma_{min}$ for the contours in Figure~\ref{fig:firstideas}(a) are displayed. But this issue can be resolved with ease by normalizing $\pi_{min}$:
\[\dot{\pi}_{min}(\vartheta) = \frac{\pi_{min}(\vartheta)}{\max_{t \in \Theta} \pi_{min}(t)}, \quad \vartheta \in \Theta.\]    
Note that this normalization step does not compromise the validity of $\pi_{min}$ since $\dot{\pi}_{min}$ is an inflated version of it. This is illustrated in Figure~\ref{fig:firstideas}(b).  

\vspace{2mm}


\subsection{A general solution}
\label{ss:general}
The solution proposed above, in terms of minimum of the $k$ IMs possibility contours, while intuitive, isn't the exclusive option available. In fact, any continuous function from $[0,1]^k$ to $\mathbb{R}$ that is monotonic in each coordinate can be used to merge the $k$ contours. This chosen function can then be: i) validified, by leveraging the fact that the $k$ contours are iid  $\unif(0,1)$; and ii) normalized, if the validified function doesn't reach the value of one. The outcome is a fused IMs possibility contour that is provably valid. I formally present this general solution next, in the form of a theorem.

\begin{theorem}
Assume $\pi_1, \ldots, \pi_k$ are $k$ independent and exact valid IMs possibility contours for a common unknown quantity of interest $\theta$. Let $\gamma_c$ be any continuous function from $[0,1]^k$ to $\mathbb{R}$ that is monotonic in each coordinate and $F$ be the cumulative distribution function of $\gamma_c(U^k)$ for $U^k= (U_1, \ldots, U_k)$ iid $\unif(0,1)$. Then 
\begin{equation}\label{eq:normGeneral}
\dot{\pi}_c(\vartheta) = \frac{\pi_c(\vartheta)}{\max_{t \in \Theta} \pi_c(t)} \quad \vartheta \in \Theta, 
\end{equation}
where
\[\pi_c(\vartheta) = F\left(\gamma_c(\pi_1(\vartheta), \ldots, \pi_k(\vartheta))\right), \] 
is a valid fused IMs' possibility contour in the sense that
\begin{equation}\label{eq:valdotpi}
\mathbb{P}_{U^k}(\dot{\Pi}_c(\theta) \leq \alpha)\leq \alpha, \quad  \text{for all $\alpha \in [0,1]$}.    
\end{equation}
\end{theorem}
\begin{proof}
    Since $\pi_1, \ldots, \pi_k$ are independent and exact valid IMs' possibility contours, $\Pi_1(\theta), \ldots, \Pi_k(\theta)$ are iid $\unif(0,1)$ and $\Pi_c(\theta)\sim \unif(0,1)$. The normalization in \eqref{eq:normGeneral} leads to $\dot{\pi}_c(\vartheta) \geq \pi_c(\vartheta)$ for all $\vartheta \in \Theta$, thereby yielding \eqref{eq:valdotpi}.
\end{proof}

Two remarks are in order regarding the choice of $\gamma_c$. First, despite Theorem~1 stating that any monotone function can be used to fuse IMs, as all $k$ contours provide reliable information about $\theta$ I follow \citet{Dubois2001} and recommend conjunctive or statistical combinations. Second, 
from a computational standpoint, the normalization in \eqref{eq:normGeneral} would probably be non-trivial if $\theta$ is vector-valued. Consequently, it may be preferable to opt for a $\gamma_c$ whose validified form already constitutes a possibility contour.

As a final illustration, consider the five independent contours depicted in Figure~\ref{fig:general}(a), constructed for making inferences on some population mean $\theta$. Figure~\ref{fig:general}(b) displays three fused IMs: one based on $\gamma_{min}$ (red), another based on $\gamma_c$ obtained by multiplying the initial contours (blue), aligning with Fisher's suggestion for combining p-values \citep{Fisher32}, and a third that takes $\gamma_c$ as the average of the initial contours (green). The IM that takes into account the knowledge that data comes from a normal distribution and also assumes access to the original data is displayed in black. This is the best possible IM for inferences on $\theta$ \citep{SyringMartin19}. It is, of course, much more efficient than the three fused IMs considered. However, it's important to recall that the fused IMs lack knowledge about normality, original data or even the sample sizes, but are still valid solutions.


\begin{figure}[t]
\begin{center}
\subfigure[IMs' contours to be fused]{\scalebox{0.3}{\includegraphics{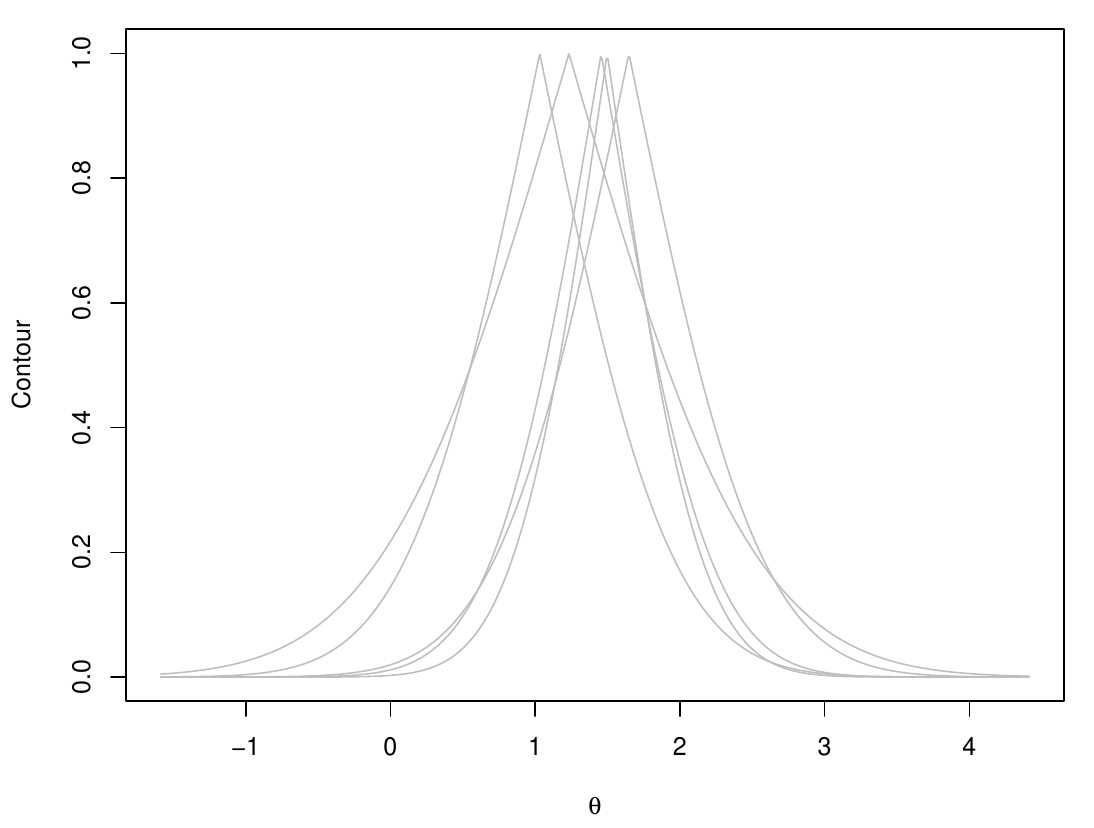}}}
\subfigure[Fused IMs and optimal IM]{\scalebox{0.3}{\includegraphics{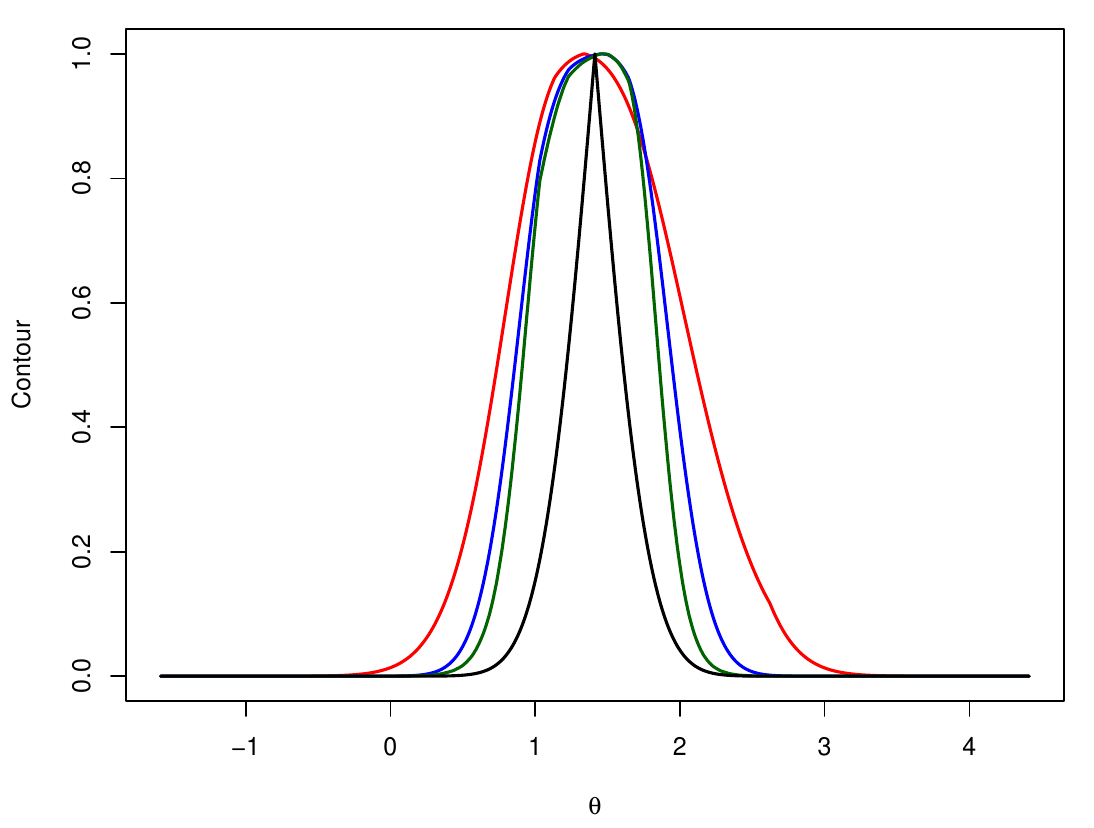}}}
\end{center}
\caption{Details for the example in Section~\ref{ss:general}. }
\label{fig:general}
\end{figure}

\section{Conclusion and remarks}\label{S:Conclusion}
This paper introduces a general framework for fusing independent IMs when no information regarding the construction of each individual IM is available. 
The underlying logic of this framework mirrors the core rationale of the IMs approach, making it akin to an application of IMs to problems where data consists of IMs' possibility contours themselves. This framework has close connections with methodologies to combine p-values in the statistics literature and methodologies to combine possibility measures in the imprecise probability literature. In summary, all three frameworks begin with a fusing function $\gamma_c$, but while p-values fusing rules simply validify
$\gamma_c$
  and contour fusing rules solely normalize 
$\gamma_c$, IMs fusing rules first validifies
$\gamma_c$ and then normalizes it.

This section concludes with several remarks and some future directions. First, even though the proposed framework does not require knowledge of how the individual IMs were constructed, it is not uncommon for information regarding the sample sizes to be available. There is a compelling rationale to integrate weights for the different contours in such instances, in the hope that the solution is more efficient. This will be explored elsewhere. Second, the normalization in \eqref{eq:normGeneral}, despite intuitive, is not the only option to guarantee the solution is a possibility measure. For alternatives and the reasoning behind them, see \citet{CELLA20221}.
Third, it has been suggested that the validification step should precede the normalization step.  But one may wonder if normalization could be conducted first. 
After all, this is the approach typically taken when selecting a function like the likelihood ratio in traditional IMs construction.  However, in the present context, if normalization is done first, one can no longer leverage the uniform nature of the individual contours to perform the validification step. Fourth, all the derivations above assumed that the IMs' contours to be combined are exact valid, but, in some applications, it is possible that some of them are just valid. Since these ``just valid'' contours are ``larger'' than necessary, when fusing the IMs' contours we obtain an inflated $\gamma_c$. Importantly, the validification step above still outputs a valid combination, albeit less efficient. Lastly,  
efficiency and optimality evaluations of different solutions, as well as the problem of fusing dependent IMs will be explored in future work.


\begin{credits}
\subsubsection{Acknowledgements}
The author thanks Professor Ryan Martin for his helpful feedback on an earlier version of this manuscript.

\subsubsection{\discintname}
The author has no competing interests to declare that are relevant to the content of this article. 
\end{credits}
%
%
%
\bibliographystyle{apalike}
\bibliography{literature.bib}

\newcommand{\noop}[1]{}
\begin{thebibliography}{}

\bibitem[Cella and Martin, 2022]{CELLA20221}
Cella, L. and Martin, R. (2022).
\newblock Valid inferential models for prediction in supervised learning problems.
\newblock {\em International Journal of Approximate Reasoning}, 150:1--18.

\bibitem[Cousins, 2008]{cousins2008annotated}
Cousins, R.~D. (2008).
\newblock Annotated bibliography of some papers on combining significances or p-values.
\newblock {\tt arXiv:0705.2209}.

\bibitem[Dubois et~al., 2016]{dubois:hal-01484952}
Dubois, D., Liu, W., Ma, J., and Prade, H. (2016).
\newblock {The basic principles of uncertain information fusion. An organised review of merging rules in different representation frameworks.}
\newblock {\em {Information Fusion}}, 32:12--39.

\bibitem[Dubois and Prade, 1988]{DuboisPrade88}
Dubois, D. and Prade, H. (1988).
\newblock Representation and combination of uncertainty with belief functions and possibility measures.
\newblock {\em Computational Intelligence}, 4(3):244--264.

\bibitem[Dubois and Prade, 2001]{Dubois2001}
Dubois, D. and Prade, H. (2001).
\newblock Possibility theory in information fusion.
\newblock In Della~Riccia, G., Lenz, H.-J., and Kruse, R., editors, {\em Data Fusion and Perception}, pages 53--76, Vienna. Springer Vienna.

\bibitem[Dubois et~al., 1999]{Dubois1999}
Dubois, D., Prade, H., and Yager, R. (1999).
\newblock Merging fuzzy information.
\newblock In Bezdek, J.~C., Dubois, D., and Prade, H., editors, {\em Fuzzy Sets in Approximate Reasoning and Information Systems}, pages 335--401, Boston, MA. Springer US.

\bibitem[Fisher, 1932]{Fisher32}
Fisher, R.~A. (1932).
\newblock {\em Statistical Methods for Research Workers}.
\newblock 4th ed. Oliver and Boyd, Edinburgh.

\bibitem[Martin, 2021]{imprecisefrequentist}
Martin, R. (2021).
\newblock An imprecise-probabilistic characterization of frequentist statistical inference.
\newblock {\tt arXiv:2112.10904}.

\bibitem[Martin, 2022a]{ryanpp1}
Martin, R. (2022a).
\newblock Valid and efficient imprecise-probabilistic inference with partial priors, i. first results.
\newblock {\tt arXiv:2203.06703}.

\bibitem[Martin, 2022b]{ryanpp2}
Martin, R. (2022b).
\newblock Valid and efficient imprecise-probabilistic inference with partial priors, ii. general framework.
\newblock {\tt arXiv:2211.14567}.

\bibitem[Martin and Liu, 2015]{martinbook}
Martin, R. and Liu, C. (2015).
\newblock {\em Inferential Models: Reasoning with Uncertainty}.
\newblock Monographs in Statistics and Applied Probability Series. Chapman \& Hall/CRC Press.

\bibitem[Martin and Syring, 2019]{SyringMartin19}
Martin, R. and Syring, N. (2019).
\newblock Validity-preservation properties of rules for combining inferential models.
\newblock In De~Bock, J., de~Campos, C.~P., de~Cooman, G., Quaeghebeur, E., and Wheeler, G., editors, {\em Proceedings of the Eleventh International Symposium on Imprecise Probabilities: Theories and Applications}, volume 103 of {\em Proceedings of Machine Learning Research}, pages 286--294. PMLR.

\bibitem[Oosterhoff, 1976]{OosterhoffPvalues}
Oosterhoff, J. (1976).
\newblock {\em Combination of one-sided statistical tests}.
\newblock MC Tracts.

\bibitem[Owen, 2009]{OwenPvalue}
Owen, A.~B. (2009).
\newblock Karl pearson's meta-analysis revisited.
\newblock {\em The Annals of Statistics}, 37(6B):3867--3892.

\end{thebibliography}

\end{document}